\documentclass[preprint,11pt]{elsarticle}

\usepackage[T1]{fontenc}
\usepackage[colorlinks=true,citecolor=blue]{hyperref}
\usepackage{mathptmx, amsmath, amssymb, amsfonts, amsthm, mathptmx, enumerate, color}
\usepackage{mathbbol,bm,tensor}
\usepackage{booktabs}
\usepackage{cases}
\usepackage{mathrsfs}
\usepackage{stmaryrd}
\usepackage{amsthm}

\allowdisplaybreaks[3]

\newtheorem{theorem}{Theorem}[section]
\newtheorem{corollary}{Corollary}[section]
\newtheorem{lemma}{Lemma}[section]

\newtheorem{definition}{Definition}[section]

\numberwithin{equation}{section}
\begin{document}
	
	\begin{frontmatter}

		\title{On Spectral Invariant Dense Subalgebras of Uniform Roe Algebras with Subexponential Growth}
		
		\author[a]{Siqi Jiang} %% Author name
		\ead{jiangsiqi990819@163.com}
		
		\author[b]{Xianjin Wang \corref{cor1}}
		\ead{xianjinwang@cqu.edu.cn}

		%% Author affiliation
		\address[a]{College of science,
			Civil Aviation Flight University of China,
			Guang Han 618307, P. R. China}
		
		%% Author affiliation
		\address[b]{College of Mathematics and Statistics
			Chongqing University, (at Huxi Campus),
			Key Laboratory of Nonlinear Analysis and its Applications (Chongqing University),
			Ministry of Education,
			Chongqing 401331, P. R. China.}
		\cortext[cor1]{Corresponding author}

		\begin{abstract}
			In this paper, we study spectrally invariant subalgebras of uniform Roe algebras for discrete groups with subexponential growth. For a group \( G \) with subexponential growth and satisfying property \( P \), we construct a class of subalgebras \( R^{\infty}(G) \). We then prove their spectral invariance in \( C_u^*(G) \) through the application of admissible weights.  This extends \(\ell^2\)-norm spectral invariance results beyond polynomial growth settings.
		\end{abstract}

		%\Keywords{Uniform Roe algebra, Spectral subalgebra, Polynomial growth}        % the keywords
		%\MRSubClass{46L05, 47L40, 46L80 }      % MR(2000) Subject Classification
	\end{frontmatter}
	
	\section{Introduction } 
	The uniform Roe algebra is a geometric $C^{*}$-algebra. Its spectral invariant subalgebra plays an important role in calculating K-theory groups, verifying Baum-Connes conjecture and studying Fredholm indices of geometric operators. \cite{GWY2008,NVG2002,NJ2000,1Y1995}.
	If we find the spectral invariant subalgebra of the uniform Roe algebra under the $\ell^2$-norm, we can calculate its $K$-theory using the cyclic homology theory  \cite{A1985}.
	%The Baum-Connes conjecture provides a method for calculating the $K$-theory groups of uniform Roe algebras. 
	%Specifically, if we find a spectral invariant dense subalgebra of  Uniform Roe algebra, then we can compute its $K$-theory using cyclic homology theory \cite{A1985}.
	
	Consequently, the construction of spectral invariant dense subalgebras of uniform Roe algebras has garnered considerable attention recently \cite{ER2009,GKM2006,KK2010,KM2006,N1999,LB1992}.
	For finitely generated groups with polynomial growth under the $\ell^1$-norm, Fendler, Gr\"{o}chenig, and Leinert \cite{GKM2008} obtained that the Wiener algebra \( W(G) \) forms a spectrally invariant subalgebra of the uniform Roe algebra \( C_u^*(G) \). More results about spectral subalgebras under the $\ell^1$-norm can be referred to \cite{AG1997,GKM2006,K2004,KM2006,KW2004}. However, under the $\ell^
	2$-norm, obtaining results analogous to those in \cite{GKM2008} for the $\ell^1$-norm poses significant challenges due to fundamental differences in norm structures.
	Notably, under the $\ell^2$-norm, some results have also been obtained. 
	Chen and Wei \cite{CW2003} demonstrated that for commutative \( C^*\)-algebras \( \mathcal{B} \) with group actions, the Schwartz function space constitutes a spectrally invariant dense subalgebra of the reduced crossed product \textit{if and only if} \( G \) exhibits polynomial growth. This result is an extension of the $\ell^2$-norm of the Wiener algebra results of Fendler et al. \cite{GKM2008} in the $\ell^1$-framework, and provides a methodological basis for subsequent research. 
	For countable discrete groups with polynomial growth, Chent al. \cite{CWW2015} constructed the weighted subalgebra $H_{\ell,B}^{\infty}(G)$ and proved its spectral invariance in the uniform Roe algebra  \( C_u^*(G) \), breaking through the limitation of $\ell^1$ and $\ell^2$-norm difference. Chen, Jiang, and Zhou \cite{CJZ2017} constructed the Fr\'{e}chet subalgebra \( H^\infty_l(G) \) with spectral invariance in \( \mathcal{A}_u(G) \), specifically for discrete groups satisfying the rapid decay (RD) property. 
	
	Due to the weakening of the attenuation constraint, the methods used in the above-mentioned polynomial growth groups are not applicable to the subexponential growth groups under the $\ell^2$-norm. This promotes the development of new analytical methods for subexponential growth groups.
	In \cite{KB2007}, Gr\"{o}chenig and Ziemowit studied Banach algebras of pseudo-differential operators and their almost diagonalized properties on Abelian groups, particularly $\mathbb{Z}^d$. Its weight conditions can be extended to the sub-exponential growth group, providing a new perspective for the spectral analysis of \( C_u^*(G) \).
	Concurrently, Sun \cite{S2007,S2011} employed admissible weights to investigate  the non-commutative inverse closed subalgebras of infinite-dimensional matrix algebras, providing a tool independent of decay conditions for the construction of subalgebras of  \( C_u^*(G) \) on subexponential groups.

	Inspired by the above-mentioned research, this paper constructs a class of  subalgebras of the uniform Roe algebra for countable discrete groups with subexponential growth, employing subexponential growth weights.
	We further proved that these subalgebras are spectrally invariant within the uniform Roe algebra by using admissible weights and growth conditions satisfying property $P$.

	\section{Preliminaries } 
	
	In this section, we briefly review some notations and preliminary results  needed in the sequel.
	
	\begin{definition}
		Let $G$ be a countable group and  $l$ be a proper length function on $G$. For $\tau \in \left [1,\infty  \right )$, let $\left | B(x, \tau)\right |$ denote the number of elements in the ball $B(x,\tau) =\left \{ y\in G: \rho_{l}(x,y)<\tau \right \} $.  
		We say $G$ has {\it subexponential growth} if
		\begin{equation}\label{subgrow}
			\lim_{\tau  \to +\infty} \frac{\ln(\sup_{x\in G}\left | B(x,\tau ) \right | )}{\tau} =0.
		\end{equation}
	\end{definition}
	\begin{definition}
		Let $(G,\rho _{l})$ be defined as above. The precompleted uniform Roe algebra of $G$ is defined to be
		\[C_{u}[G]=\left \{ T:G\times G\to \mathbb{C}\mid T~\text{is~bounded~and~finitely~propagated}   \right \},\]
		which is a $*$-subalgebra of $\mathcal{B}(\ell ^{2}(G))$. 
		
		Its operator norm closure is called the {\it uniform Roe algebra} of $G$, denoted by $C_{u}^{*}(G)$, i.e.,
		\[C_{u}^{*}(G)=\overline{C_{u}[G]} ^{\left \| \cdot  \right \|_{\mathcal{B}(\ell ^{2}(G))}}.\]
	\end{definition}
	\sloppy
	\begin{definition}
		For an operator $T=[t(x,y)]_{(x,y)\in G\times G} \in \mathcal{B} (\ell ^{2}(G))$, let $f:G\to \left [ 0,\infty \right )$ be the function defined by $f(z)=\sup _{\left \{x,y\in G,y^{-1}x=z \right \} } \left | t(x,y) \right | $  for all $z\in G$. We call $f$ the {\it dominating vector} of $T$.
		
		%Here a function $g$ on $G \times G$ is said to be left-invariant if $g(zx, zy) = g(x, y)$ for all $x, y, z \in G $.
		%	Clearly a left-invariant function $g$ on $G \times G$ is determined by a function $\tilde{g}$ on $G$
		%	by $\tilde{g}(x)=\sup _{\left \{x,y\in G,y^{-1}x=g \right \} }\left | g(x,y) \right |$, $x, y \in G$.
	\end{definition}
	%In this paper, we usually let $\tilde{a}$, $\tilde{b}$ and $\tilde{c}$ be the {\it dominating vectors} of operators $A$, $B$ and $C$ $\in \mathcal{B}(\ell ^{2}(G))$ respectively. Specifically, $\tilde{c}(z)=\sup_{\left \{ x,y,z\in G:y^{-1}x=z \right \} }$ $ \left | c(x,y) \right |.$ 
	\begin{definition}
		A positive symmetric measurable function $w$ on $G\times G$  is called a weight, if it fulfills
		\[1 \le  w \left( x,y \right )= w (y,x) \le \infty  \quad \text{for~all}~ x,y\in G;\]
		\[D(w):=\sup_{x\in G}w (x,x)<\infty ;\]
		\[\sup_{\rho_{l} (x,\tilde x)+\rho_{l}(y,\tilde y)\le C_{0} } \frac{w (x, y)}{w (\tilde x,\tilde y)}\le D(C_{0},w )<\infty \quad  \text{for~all}~ C_0\in (0,\infty),\]
		where $D(w)$ and $D(C_0, w)$ are positive constants associated with $w$.
		
	\end{definition}
	
	\begin{definition}
		Let $1\le p,r\le \infty$. We say that a weight $\omega $ is $\left (p,r  \right ) $-{\it admissible} if there exist another weight $v$ and two positive constants $D\in \left ( 0,\infty  \right ) $ and
		$ \theta \in \left ( 0,1 \right ) $ such that
		\begin{equation}\label{w1}
			w\left ( x,y \right ) \le  D\left ( w\left ( x,z \right ) v\left ( z,y \right ) +v\left ( x,z \right ) w\left ( z,y \right )  \right )~\text{for~all}~ x,y,z\in G,
		\end{equation}
		\begin{equation}\label{w22}
			\sup_{x\in G}\left \| (vw^{-1})\left ( x,\cdot \right ) \right \|_{p^{_{'}}} +\sup_{y\in G}\left \|( vw^{-1})\left ( \cdot,y \right ) \right \| _{p^{_{'}}}\le D,
		\end{equation}
		and
		\begin{equation}\label{w2}
			\inf _{\tau >0}a_{r' }\left ( \tau  \right ) +b_{p' }\left ( \tau  \right )t                                                                                                                                                                                                                                                                                                                                                                                                                                                                                                                                                                                                                                                                                                                                                                                                                                                                                                                                                                                                                                                                                                                                                                                                                                                                                                                                                                                                                                                                                                                                                                                                                                                                                                                                                                                                                                                                                                                                                                                                                                                                                                                                                                                                                                                                                                                                                                                                                                                                                                                                                                                                                                                                                                                                                                                                                                                                                                                                                                                                                                                                                                                                                                                                                                                                                                                                                                                                                                                                                                                                                                                                                                                                                                                                                                                                                                                                                                                                                                                                                                                                                                                                                                                                                                                                                                                                                                                                                                                                                                                                                                                                                                                                                                                                                                                                                                                                                                                                                                                                                                                                                                                                                                                                                                                                                                                                                                                                                                                                                                                                                                                                                                                                                                                                                                                                                                                                                                                                                                                                                                                                                                                                                                                                                                                                                                                                                                                    
			\le Dt^\theta ~ \text{for~all}~ t\ge 1,
		\end{equation} 
		where ${p}' =p/\left(p-1 \right)$, $r' =r/\left ( r-1 \right )$,
		\begin{equation*}\label{w3}
			a_{r'}(\tau)=\sup_{x\in G}\left \| v\left ( x,\cdot \right )\chi _{B\left ( x,\tau  \right )\left ( \cdot  \right )  }  \right \|_{r^{'}} +\sup_{y\in G}\left \| v\left ( \cdot,y \right )\chi _{B\left (y, \tau  \right )\left ( \cdot  \right )  } \right \|_{r^{'}},
		\end{equation*}
		\begin{equation*}
			b_{p^{'}}(\tau)=\sup_{x\in G}\left \| (vw^{-1})\left ( x,\cdot \right )\chi _{X\setminus B\left ( x,\tau  \right )\left ( \cdot  \right )  }  \right \|_{p^{'}} +\sup_{y\in G}\left \|( vw^{-1})\left ( \cdot,y \right )\chi _{X\setminus B\left (y, \tau  \right )\left ( \cdot  \right )  } \right \|_{p^{'}},
		\end{equation*}
		$\chi _{E}$ is the characteristic function on the set $E$, and $\left \| \cdot  \right \| _{p}$ is the norm on $\ell^{p}$.
		
		Unless stated otherwise, in this paper, $p=2$.
	\end{definition}
	
	%Note that any countable discrete group admits a proper length function $l $ on G, which  makes $(G,\rho_l) $ a metric space.
	
	%Then we review the definitions of polynomial, subexponential growth, and uniform Roe algebra.
	%\begin{definition}
	%Let $\mathcal{A}  \subseteq \mathcal{B} $ be two Banach algebras with a common identity. Then $\mathcal{A}$ is called inverse-closed in $\mathcal{B}$ if
	%\[a\in \mathcal{A} ~\text{and}~ a^{-1} \in \mathcal{B} \Rightarrow  a^{-1}  \in \mathcal{A}.\]
	%\end{definition}
	%Let $\mathcal{A}$ be a Banach algebra with an identity element $e$, and let $a\in \mathcal{A}$. The spectrum of $a$ is 
	%\[\sigma _\mathcal{A}(a)=\left \{ \lambda \in\mathbb{C}:a-\lambda e~ \text{is not invertible in}~\mathcal{A} \right \}.\]
	%The spectral radius of $a$ is
	%\[r _\mathcal{A} (a)=\max\left \{ \left | \lambda  \right | :\lambda \in  \sigma _\mathcal{A}(a)\right \}=\lim_{n \to \infty} \left \| a^n \right \|_\mathcal{A}^\frac{1}{n}.\]
	
	%We conclude this section by recalling the Hulanicki's lemma,  which explains why the term ``spectral invariance" is often used in connection with inverse-closed subalgebras.
	
	\begin{lemma}[Hulanicki's lemma~cf.\cite{A1972}]
		Let $\mathcal{A}\subseteq \mathcal{B} $ be two Banach algebras with a common identity. Then the following statements are equivalent:
		\begin{itemize}
			\item[(1)] $\mathcal{A}$ is inverse-closed in $\mathcal{B}$;
			\item[(2)] $r _\mathcal{A}(a)=r _\mathcal{B}(b)$ for all $a=a^{*}$ in $\mathcal{A}$;
		\end{itemize}
		where
		%\[\sigma _\mathcal{A}(a)=\left \{ \lambda \in\mathbb{C}:a-\lambda e~ \text{is not invertible in}~\mathcal{A} \right \},\]
		%and
		$r _\mathcal{A} (a)=\max\left \{ \left | \lambda  \right | :\lambda \in  \sigma _\mathcal{A}(a)\right \}=\lim_{n \to \infty} \left \| a^n \right \|_\mathcal{A}^\frac{1}{n}.$
	\end{lemma}
	In this paper, we denote by $C$ and $D$  are generic constants whose value may change from line to line.
	\section{Spectral invariant subalgebras of a subexponentially growing group}
	In this section, for subexponential growth groups \( G \) with property \( P \), we construct the subalgebras by taking the union of a family of Banach algebras, and then establish their spectral invariance under \(\ell^2\)-norm.
	
	\begin{definition}[cf.~\cite{CW2003,2Y1995}] 
		We call a group $G$ with property $P$ if for any $\alpha>0$ and $0<\beta<1$, there exists $C_{\alpha ,\beta }$ such that
		\[ \left | B(x,r) \right | \le C_{\alpha ,\beta }\exp(\alpha r^{\beta })\quad\text{for~ all}~  x\in G~\text{and
		}~r>0.\]
	\end{definition}
	It follows from (\ref{subgrow}) that if $G$ satisfies {\it property P}, then $G$ has subexponential growth.
	\begin{definition}
		Let $G$ be a countable group with a proper length function $l$. If $G$ has property $P$,  the space $R_{\alpha ,\beta }(G)$ is defined as follows  
		\[R_{\alpha ,\beta }(G)= \{ T=(t(x,y)_{x,y\in G}):G\times G\longrightarrow \mathbb{C}\mid  \left \| T \right \|_{\alpha,\beta}<\infty   \}\quad\text{for}~\alpha>0,0<\beta<1.\]
		where
		\[\left \| T \right \| _{\alpha ,\beta }= [ {\sum_{z\in G} (\sup_{\left \{ x,y\in G:y^{-1}x=z \right \} }\left | t\left (x,y\right) \right | )^2   \exp(2\alpha  \rho_{l} (x,y)^{\beta }}) ]^\frac{1}{2} <\infty.\]
		We define $R^{\infty }(G)$ by
		\[R^{\infty }(G)=\bigcup _{\alpha>0,0<\beta<1 }R_{\alpha ,\beta }(G).\]
		$R^{\infty }(G)$ 
		is the algebra consisting of functions $T$ that satisfy, for some $\alpha>0$ and $0<\beta<1$,
		\begin{align*}
			\left \| T \right \| _{\alpha ,\beta }&=
			(\sum_{z\in G}  \left((\tilde{t}w)(z)\right)^{2} ) ^\frac{1}{2} = \|  \tilde {t}w  \|_2<\infty,  
		\end{align*}
		where $w(z)=w(y^{-1}x)_{\left \{x,y\in G, y^{-1}x=z \right \} }=\exp(\alpha l(z)^{\beta })$.	
	\end{definition}
	Next, we show that $R^{\infty }(G)$ is a spectral invariant subalgebra when $G$ is a {\it subexponential growth} space satisfying property $P$.

	\begin{theorem}\label{thm4.2}
		Let $G$ be a countable discrete group with a proper length function $l$. If $G$ satisfies property $P$,  the subexponential weight $w(x,y)=\exp(\alpha \rho_{l }(x,y)^{\beta}$ is a $(2,r)$-admissible weight.
	\end{theorem}
	
	\begin{proof}
		
		In order to prove $w$ is an admissible weight satisfying (\ref{w1})-(\ref{w2}), the proof is carried out in several steps.
		
		\textbf{Step 1.}  The weights $w$ and $v$ satisfy (\ref{w1}), i.e.,
		$$w(x,y)\le D\left ( w(x,z)v(z,y)+ v(x,z)w(z,y)\right)$$ for all $x,y,z\in G$ with $D=1$.
		
		\noindent We recall the form of weight $w$ and define the weight $v$ as follows
		\begin{equation}
			\label{adw1}
			w(x,y)=\exp(\alpha \rho_{l }(x,y)^{\beta})\quad \text{for~all}~ x,y\in G,
		\end{equation}
		\begin{equation}\label{adv1}
			v(x,y)=\exp(\alpha (2^{\beta }-1) \rho_{l}(x,y)^{\beta}) \quad \text{for~all}~ x,y\in G,
		\end{equation}
		with $\alpha\in (0,\infty),\beta \in (0,1)$.
		Note that the following inequality holds
		$$ 1\leq s^\beta +(2^\beta -1)(1-s)^\beta\quad\text{for~all}~\tfrac{1}{2}\leq s \leq 1.$$
		Let 
		$$
		s=\begin{cases}
			\frac{\rho_l(x,z)}{\rho_l(x,z)+\rho_l(z,y)}, &\text{ if } \rho_l(x,z)\geq \rho_l(z,y)\\
			\frac{\rho_l(z,y)}{\rho_l(x,z)+\rho_l(z,y)}, &\text{ if } \rho_l(x,z)< \rho_l(z,y)
		\end{cases}.
		$$
		If $ \rho_l(x,z)\geq \rho_l(z,y)$, we have
		\begin{align*}
			w(x,y)&=\exp(\alpha \rho_{l }(x,y)^{\beta})\le \exp(\alpha( \rho_{l }(x,z)+\rho_{l }(z,y))^{\beta }) \\
			&\leq\exp(\alpha (s^\beta+(2^\beta-1)(1-s)^\beta)( \rho_{l}(x,z)+\rho_{l}(z,y))^{\beta})\\
			&=\exp(\alpha\rho_l(x,z)^\beta +(2^\beta-1)\rho_l(z,y)^\beta)=w(x,z)v(z,y).
		\end{align*}
		Similarly, if $\rho_l(x,z)< \rho_l(z,y)$, we have
		$$w(x,y)\leq v(x,z)w(z,y)\quad\text{for~all}~x,y,z\in G.$$
		Hence,
		$$w(x,y)\leq D\left(w(x,z)v(z,y)+v(x,z)w(z,y)\right)$$
		for all $x,y,z\in G$ with $D=1$.
		Thus, the weights $w$ and $v$ satisfy (\ref{w1}).
		
		\textbf{Step2.} The weights $w$ and $v$ satisfy (\ref{w22}), i.e.,
		\begin{align*}
			\sup_{x\in G}\left \| vw^{-1}(x,\cdot ) \right \| _{2}+\sup_{y\in G}\left \| vw^{-1}(\cdot,y ) \right \| _{2}<\infty.
		\end{align*}
		Firstly, we estimate $\sup_{x\in G}\left \| vw^{-1}(x,\cdot ) \right \| _{2}$.
		
		Based on {\it property P} of the group $G$, we know that for any $0<\alpha'<\alpha(2-2^{\beta})$ and $0<\beta'<\beta<1$, there exists $C_{\alpha' ,\beta' }$ such that
		\[ \left | B(x,\tau) \right | \le C_{\alpha' ,\beta' }\exp(\alpha' \tau^{\beta' })%=C\exp(\alpha' \tau^{\beta' })
		\quad\text{for~ all}~  x\in G~\text{and
		}~\tau\ge1.\]
		By the definition of the weights $v$ and $w$ , we have
		\begin{align}\label{vw}
			vw^{-1}(x,y)=\exp(-\alpha (2-2^{\beta }) \rho_{l}(x,y)^{\beta}).
		\end{align}
		Since 
		\begin{gather}\label{2.7'}
			\sup_{x\in G}\left \| vw^{-1}\left ( x,\cdot \right )\chi _{B\left ( x,\tau  \right )\left ( \cdot  \right )  }  \right \|_{2}=\sup_{x\in G}(  (	vw^{-1}(x,y)    )^{2} \left | B(x,\tau ) \right |)^\frac{1}{2},
		\end{gather}
		and
		\begin{gather}\label{2.7}
			\sup_{x\in G}\left \| vw^{-1}\left ( x,\cdot \right )\chi _{X\setminus B\left ( x,\tau  \right )\left ( \cdot  \right )  }  \right \|_{2}=(\sum_{j=0}^{\infty } \sum _{2^{j}\tau \le\rho_l (x,y)< 2^{j+1}\tau} ( vw^{-1}(x,y)  )^{2})^\frac{1}{2}.
		\end{gather}
		It is worth noting that the above results are also true for $y$.
		
		\noindent Combing  (\ref{vw}),(\ref{2.7'}) and (\ref{2.7}), we have
		\begin{align*}%\label{sub-2.51}
			&	\sup_{x\in G}\left \| vw^{-1}\left ( x,\cdot \right )\chi _{B\left ( x,\tau  \right )\left ( \cdot  \right )  }  \right \|_{2}= \sup_{x\in G}(\sum _{\rho_{l}(x,y)< \tau }\exp(-2\alpha (2-2^{\beta })\rho_{l }(x,y)^{\beta })^\frac{1}{2}\notag\\
			&\le C\sup_{x\in G}(\exp(-2\alpha (2-2^{\beta })\tau ^{\beta})\left| B(x,\tau)\right|)^\frac{1}{2}\notag\\
			&\le C\sup_{x\in G}(\exp(-2\alpha (2-2^{\beta })\tau ^{\beta}+\alpha'\tau^{\beta'}))^\frac{1}{2}\le C
		\end{align*}
		and
		\begin{align}\label{sub-2.52}
			&\sup_{x\in G}\left \| (vw^{-1})\left ( x,\cdot \right )\chi _{X\setminus B\left ( x,\tau  \right )\left ( \cdot  \right )  }  \right \|_{2}\leq\Big(\sum_{j=0}^{\infty} (\exp (-2\alpha (2-2^{\beta})(2^{j}\tau)^{\beta})) \left| B(x,2^{j+1}\tau)\right|\Big)^\frac{1}{2} \notag \\
			&\leq C\Big(\sum_{j=0}^{\infty } \exp(-2\alpha(2-2^{\beta}) 2^{j\beta}\tau^{\beta }+\alpha'2^{(j+1)\beta'}\tau ^{\beta'} )\Big)^\frac{1}{2}\notag \\
			&\leq C\Bigl(\sum_{j=0}^{\infty }\exp( -2(\alpha(2-2^{\beta}) -\alpha')2^{j\beta}\tau^{\beta} )\Bigr)^{\frac{1}{2}}\notag \\
			&=C \Bigl(\sum_{j=0}^{\infty }\exp(-\alpha(2-2^\beta)\tau^\beta[2(1-\tfrac{\alpha'}{\alpha(2-2^{\beta})})]2^{j\beta} \Bigr)^{\frac{1}{2}} \notag\\
			&=C \Bigl(\sum_{j=0}^\infty q^{2\bigl(1-\tfrac{\alpha'}{\alpha(2-2^{\beta})}\bigr)2^{j\beta}}\Bigr)^{\frac{1}{2}}\leq C  (\sum_{j=0}^\infty q^{2^{j+1}})^{\frac{1}{2}} \notag \\
			&\leq C (\sum_{j=1}^{\infty}q^{2j}  )^{\frac{1}{2}}\leq C\sqrt{\frac{q^2}{1-q^2}}\leq C'q ,
		\end{align}
		where $q=\exp(-\alpha(2-2^{\beta})\tau^\beta)<1$, and note $\tau\geq 1$, the $\frac{1}{\sqrt{1-q^2}}$ indeed has upper bound.
		
		Then, we obtain 
		\begin{align*}
			\sup_{x\in G}\left \| vw^{-1}(x,\cdot ) \right\| _{2}\le (\ref{2.7'})+(\ref{2.7})
			&\le C+C'\sum_{j=0}^{\infty } \exp(-\alpha(2-2^{\beta}) \tau^{\beta } )<\infty.
		\end{align*}
		%By the property $P$,we know for some $\beta'<\beta$ and $\alpha'<\alpha(2-2^{\beta})$, there exists $C_{\alpha',\beta'}$ such that
		%	\[ \left | B(x,r) \right | \le C_{\alpha' ,\beta' }\exp(\alpha' r^{\beta' })\quad\text{for~ all}~  x\in G~\text{and	}~r>0.\]
		%where $d(G,\rho_l)>0$, $\beta'<\beta$ and $\alpha'<\alpha(2-2^{\beta})$.
		Similarly, we  obtain 
		\begin{equation*}
			\sup_{y\in G}\left \| vw^{-1}(\cdot,y ) \right \| _{2}<\infty.
		\end{equation*}
		Thus, we have the weights $w$ and $v$ satisfy (\ref{w22}).
		%\[\sup_{x\in G}\left \| vw^{-1}(x,\cdot ) \right \| _{2}+\sup_{y\in G}\left \| vw^{-1}(\cdot,y ) \right \| _{2}<\infty,\]
		%which means the weights $w$ and $v$ satisfy (\ref{w22}).
		
		\textbf{Step 3.}  The weights $w$ and $v$ satisfy (\ref{w2}), i.e.,
		\begin{align*}
			\inf_{\tau\ge 1} a_{r'}(\tau ) +b_{p'}(\tau ) \cdot t\le Ct^{\theta}\quad \text{for~all} ~t\ge 1~\text{and}~\theta\in (\frac{1}{3-2^{\beta}},1).
		\end{align*}
		Firstly, we get the following estimate
		\begin{align}\label{sub2.5}
			\sup_{x\in G}\left \| v\left ( x,\cdot \right )\chi _{B\left ( x,\tau  \right )\left ( \cdot  \right )  }  \right \|_{r'} 
			%&=\sup_{x\in G}\Big[ \sum_{y\in G,\rho_{l }(x,y)\le \tau } \Big ( 	\exp(\alpha (2^{\beta }-1) \rho_{l }(x,y)^{\beta})\Big )^{r'} \Big]^\frac{1}{r'}\notag\\
			&\le\sup_{x\in G}( \exp(r'\alpha (2^{\beta }-1) \rho_{l }(x,y)^{\beta})\left| B(x,\tau)\right|) ^\frac{1}{r'}\notag\\
			%&\le C\left ( \exp	\left(\alpha (2^{\beta }-1)\tau^\beta \right )\cdot\exp(\alpha'\tau ^{\beta'})^\frac{1}{r'}\right )\notag\\
			&\le C\exp(\alpha (2^{\beta }-1)\tau^\beta+\alpha'\tau^{\beta'}/r')\notag\\
			&\le C\exp(\alpha (2^{\beta }-1)\tau^\beta+\alpha'\tau^{\beta'})\le C\exp(\alpha \tau^\beta),
		\end{align}
		where $r'\ge 1$. 
		
		\noindent Then, by applying the inequalities (\ref{sub-2.52}) and (\ref{sub2.5}), we derive that
		%$$a_{r'}(\tau )\le C\exp\left(\alpha \tau^\beta\right)~\text{and}~b_{p'}(\tau )\le C\exp(-\alpha(2-2^{\beta}) \tau^{\beta } ).$$
		%Hence, we get
		\begin{equation}
			\label{e:key_exp}
			\inf_{\tau\ge 1} a_{r'}(\tau ) +b_{p'}(\tau ) t
			\le C\inf_{\tau\ge 1} [\exp(\alpha \tau^\beta) + \exp(-\alpha(2-2^{\beta}) \tau^{\beta } )\cdot t].
		\end{equation}
		%Indeed, let $\tau=1+f(t)$ with $f(t)$ determined later, it is enough to prove
		We assert that there is a function $f(t)$ satisfies that		
		\begin{align}
			\exp(\alpha(1+f(t))^\beta)&\le Ct^{1/(3-2^\beta)}\label{e3.15}\\
			\exp(-\alpha(2-2^\beta)(1+f(t))^\beta)\cdot t&\le Ct^{1/(3-2^\beta)}\label{e3.16}.
		\end{align}
		Indeed, 
		\begin{align*}
			(\ref{e3.15}) &\iff \exp(\alpha(1+f(t))^\beta)\leq \exp(1/(3-2^\beta)\ln t+\ln C)\\
			&\iff \alpha(1+f(t))^\beta \leq \frac{\ln t}{3-2^\beta}+\ln C,
		\end{align*}
		and 
		\begin{align*}
			(\ref{e3.16})&\iff t^{\frac{2-2^\beta}{3-2^\beta}}\leq C\left[ \exp(\alpha(1+f(t))^\beta)\right]^{2-2^\beta}\\
			&\iff \exp\left(\frac{1}{3-2^\beta}\ln t\right)\leq C^{\frac{1}{2-2^\beta}}\exp(\alpha(1+f(t))^\beta)\\
			&\iff \frac{1}{3-2^\beta}\ln t\leq \alpha(1 + f(t))^\beta +\frac{1}{2-2^\beta}\ln C.
		\end{align*}
		Setting $\frac{1}{3-2^\beta}\ln t= \alpha(1 + f(t))^\beta$, i.e., 
		$$f(t)= \left(\frac{\ln t}{\alpha(3-2^\beta)}\right)^{\frac{1}{\beta}}-1,$$
		the desired inequalities (\ref{e3.15}) and (\ref{e3.16}) hold.
		Let $\tau=1+f(t)$ in (\ref{e:key_exp}), we get 
		\[\inf_{\tau\ge 1} a_{r'}(\tau ) +b_{p'}(\tau ) \cdot t\le Ct^{ 1/(3-2^{\beta })}\le Ct^{\theta}\quad \text{for~all}~t\ge 1~,\theta\in (\frac{1}{3-2^{\beta}},1),\]
		which means the weights $w$ and $v$ satisfy (\ref{w2}).
	\end{proof}
	\begin{theorem}\label{banach-2}
		Assume that $G$ satisfies property $P$ with a proper length function $l$, $R_{\alpha,\beta }(G)$ is a Banach algebra with norm $\left \| \cdot \right \|_{\alpha,\beta }$.    
	\end{theorem}
	\begin{proof}
		%The closedness under addition is clear.
		%In order to get $R_{\alpha,\beta }(G)$ is a Banach algebra, the proof is carried out in two parts.
		
		The verification that $R_{\alpha,\beta }(G)$ forms a Banach algebra is carried out in two parts: completeness and multiplicative structure.
		
		\textbf{Claim 1.} $R_{\alpha,\beta }(G)$ is a Banach space. 
		
		We rewrite the definition of $\left \| \cdot  \right \| _{\alpha,\beta }$ as follows in the following paper  $$ \tilde{t}(z)=\sup_{ y^{-1}x=z } 
		\left | t(x,y) \right |~\text{and}~ w(z)=\exp(\alpha \rho_{l }(z)^{\beta}.$$
		The norm $\left \| \cdot  \right \|_{\alpha,\beta }$ in the following paper is defined as $\left \| T \right \|_{\alpha,\beta} 
		= [\sum _{z\in G} ((\tilde{t}  w)(z))^{2} ] ^\frac{1}{2}.$
		It is obvious that  $R_{\alpha,\beta }(G)$ 
		is a normed space with the norm $\left \| \cdot \right \|_{\alpha,\beta }$.\\
		Let $\left \{ T_{n} \right \}_{n=1}^{\infty} $ be the Cauchy sequence in $R_{\alpha,\beta }(G)$ , i.e.,$\forall s>0,\forall \varepsilon >0,\exists M>0$, \text {s.t.} $\forall m,n\ge M,$ one has $\left \| T_{m}-T_{n} \right \| _{\alpha,\beta}<\varepsilon $, that is
		\begin{align}\label{complete}
			[ \sum_{z\in G} (\sup_{y^{-1}x=z}\left | t_{m}(x,y)-t_{n}(x,y) \right | )^{2}\exp(2\alpha \rho_{l }(z)^{\beta} ]^{\frac{1}{2} }<\varepsilon,
		\end{align}
		which indicates that
		\[\sup_{y^{-1}x=z}\left | t_{m}(x,y)-t_{n}(x,y) \right | <\varepsilon\quad \text{for~all}~n,m\ge M.\]
		%Thus,
		%\[\Big |\sup_{y^{-1}x=g} t_{m}(x,y)-\sup_{y^{-1}x=g} t_{n}(x,y) \Big |<\sup_{y^{-1}x=g}\left | t_{m}(x,y)-t_{n}(x,y) \right | <\varepsilon.\]
		Therefore, the sequence $\left \{ t_n(x,y) \right \}_{n=1}^{\infty} $ is a Cauchy sequence and $\lim_{n \to \infty} t_{n}(x,y)=t_{0}(x,y)$. 
		For simplify, we denote $T_{0} = [t_{0}(x, y)]_{x,y\in G}$, where $t_{0}(x,y)=\lim_{n \to \infty} t_{n}(x,y)$.
		For any finite subset $F$ of $G$, by using (\ref{complete}), we have
		\[ \sum_{z\in F} (\sup_{y^{-1}x=z}\left | t_{m}(x,y)-t_{n}(x,y) \right | )^{2}\exp(2\alpha \rho_{l }(z)^{\beta}<\varepsilon^2\quad \text{for~all} ~n,m\ge M.\]
		Letting $n\to \infty $, and then letting $F\to G$, we get 
		\[ \sum_{z\in G} (\sup_{y^{-1}x=z}\left | t_{m}(x,y)-t_{0}(x,y) \right | )^{2}\exp(2\alpha \rho_{l }(z)^{\beta}<\varepsilon^2.\]
		i.e.,$\left \| T_{m}-T_{0} \right \| _{\alpha,\beta }<\varepsilon$.
		We have $T_{m}\to T_{0}$ and $T_m-T_0\in R_{\alpha,\beta }(G)$, which means that $T_0\in R_{\alpha,\beta }(G)$.
		Thus, the completeness holds, and $R_{\alpha,\beta }(G)$ is a Banach space.
		
		\textbf{Claim 2.} $R_{\alpha,\beta }(G)$ is a Banach algebra.
		
		It is clear that $R_{\alpha,\beta }(G)$ is an algebra. We just show that $\left \| AB \right \|_{\alpha,\beta }\le \left \| A \right \|_{\alpha,\beta }\left \| B \right \|_{\alpha,\beta }$.
		Let $A,B\in R_{\alpha,\beta }(G) $.
		We denote  $A=(a(x,y))_{x,y\in G}$, $B=(b(x,y))_{x,y\in G}\in R_{\alpha,\beta }(G)$, and write $AB=(c(x,y))_{x,y\in G}$.
		Taking $z = y^{-1}x$ for any $  x, y \in  G$, 
		we have 
		\begin{align*}
			\left \| A\right \| _{\alpha,\beta} &%
			= [\sum _{z\in G} 
			(\tilde {a}(z)w(z))^{2}  ]^\frac{1}{2}.
		\end{align*}
		Also we define $\left \| B\right \| _{1 ,v }= {\sum_{z\in G}  \tilde {b}\left (z \right )v(z)}$.
		We can write $\left \| AB \right \|_{\alpha,\beta }= \|  \tilde {c}w \|_{2}$, where
		\begin{align*}
			\tilde{c}(z)&=\sup_{_{ y^{-1}x=z } }  |c(z,I)  |=\sup_{z\in G } |\sum_{x\in G}a(z,x)b(x,I)  |\le \sum_{x\in G}\tilde{a}(x^{-1}z)\tilde{b}(x),
		\end{align*}
		where $I$ is the identity element of the discrete group $G$.\\
		Since $w$ is an admissible weight, we obtain
		\begin{align*}
			w(z,I)&\le D(w(z,x)v(x,I)+v(z,x)w(x,I))=D\left(w(x^{-1}z)v(x)+v(x^{-1}z)w(x)\right),
		\end{align*}
		where $D\in (0,\infty)$.
		
		Thus, we have the following eatimate
		\begin{align*}
			\left \| AB \right \| _{\alpha,\beta}^{2}%&=\sum_{z\in G }\left | \tilde {c}(z)w(z) \right |^2
			&\le \sum_{z\in G } |\sum_{x\in G} \tilde {a}(x^{-1}z)\tilde {b}(x)w(z)  |^2\\
			&\le D^2\sum_{z\in G } |\sum_{x\in G} \tilde {a}(x^{-1}z)\tilde {b}(x)(w(x^{-1}z)v(x)+v(x^{-1}z)w(x))  |^2\\
			%&=D^2\sum_{z\in G }| \left ( \tilde {a}w \right ) \ast   \tilde {b}v  (z)+\left ( \tilde {a}v \right ) \ast   \tilde {b}w  (z) |^2\\
			&\le 2D^2 (  |\sum_{z\in G }\left ( \tilde {a}w \right ) \ast   \tilde {b}v  (z) |^2+\ |\sum_{z\in G }\left ( \tilde {a}v \right ) \ast   \tilde {b}w (z) |^2 ).
		\end{align*}
		Let $f=\tilde{a}w$, $g=\tilde{b}v$. We know $f,g\in \ell^2(G)$. It follows from the Young's Inequality that 
		\begin{align}\label{3.3}
			\left \| AB \right \| _{\alpha,\beta}^{2}&\leq 2D^2(\ \| \tilde{b}v\ \|_1^2 \left \| \tilde{a}w \right \|_2^2 + \left \| \tilde{a}v \right \|_1^2 \ \| \tilde{b}w \ \|_2^2)\notag \\
			&= 2D^2( \left \| B\right \|_{1,v}^{2}\left \| A \right \|_{\alpha,\beta}^{2} +\left \| A \right \|_{1,v}^{2}\left \| B \right \|_{\alpha,\beta}^{2} ) . 
		\end{align}
		By utilizing the Cauchy-Schwarz Inequality and (\ref{w22}), we have
		\begin{align}\label{estm3}
			\left \| A \right \|_{1,v}
			&\le  [\sum_{z\in G} (\tilde{a}(z)w(z))^{2} ]^{\frac{1}{2}}\cdot[\sum_{z\in G} \left(vw^{-1}(z) \right)^2 ]^{\frac{1}{2}}\le D \left \| A \right \| _{\alpha,\beta }.
		\end{align}
		Combining the estimates (\ref{3.3}) and (\ref{estm3})  leads to
		\begin{align*}
			&\left \| AB \right \|_{\alpha,\beta }^2 
			\le D\left \| A \right \|_{\alpha,\beta}^2\left \| B \right \|_{\alpha,\beta }^2,
		\end{align*}
		which means $\left \| AB \right \|_{\alpha,\beta }\le D\left \| A \right \|_{\alpha,\beta } \left \| B \right \|_{\alpha,\beta } $.
		Let $\left \| \cdot  \right \|'=D \left \| \cdot  \right \| _{\alpha,\beta }$, we have  $\left \| AB  \right \|'\le \left \| A \right \|'\left \| B  \right \|'$. Hence, $R_{\alpha,\beta }(G)$  is a Banach algebra with norm $\left \| \cdot \right \|_{\alpha,\beta }$.
		
	\end{proof}
	\begin{theorem}\label{subcontain}
		Let $\left ( G,\rho_{l }  \right )$ be a discrete metric space. If $G$ satisfies property P, then $R^{\infty }(G) \subseteq C_{u}^{*}(G)$.      
	\end{theorem}
	\begin{proof}
		For any $\xi\in \ell^{2}(G) $ and $\phi \in R^{\infty }(G)$, we have the following estimation
		\begin{align*}
			\begin{split}
				\left \| \phi \xi  \right \|^{2} %&= \Big( \sum _{x} \Big | \sum_{y}\phi (x,y)\xi (y) \Big | ^{2}\Big)\\
				%& =\Big( \sum _{x} \Big | \sum_{y}\phi (x,y)exp(\alpha l(y^{-1}x))^{\beta }\exp(-\alpha l(y^{-1}x))^{\beta }\xi (y) \Big | ^{2}\Big)\\
				%&\le  \sum _{x} (\sum_{y}\left |\phi (x,y) \right | ^{2}\exp(2\alpha l(y^{-1}x))^{\beta }) (\sum_{y} exp(-2\alpha l(y^{-1}x))^{\beta }\left | \xi (y) \right | ^{2})\\
				&\le  \sum _{x} \sum_{y}\left |\phi (x,y) \right | ^{2}\exp(2\alpha l(y^{-1}x)^{\beta }) \sum_{y} \exp(-2\alpha l(y^{-1}x)^{\beta })\left | \xi (y) \right | ^{2}\\
				&\le \sup _{x} \sum_{y}\left |\phi (x,y) \right | ^{2}\exp(2\alpha l(y^{-1}x)^{\beta })\sum_{x} \sum_{y} \exp(-2\alpha l(y^{-1}x)^{\beta })\left | \xi (y) \right | ^{2}.
			\end{split}
		\end{align*}
		For the term $\sum_{x} \sum_{y} \exp(-2\alpha l(y^{-1}x)^{\beta })\left | \xi (y) \right | ^{2}$, we obtain the following inequality, for some $0<\alpha^{'}<\alpha$ and $0<\beta^{'}<\beta<1 $,
		\begin{align*}
			\begin{split}
				&\sum_{x} \sum_{y} \exp(-2\alpha l(y^{-1}x)^{\beta })\left | \xi (y) \right | ^{2}
				\le \left \| \xi  \right \|^{2}\sum_{n=0}^{\infty } \sum _{n\le l(y^{-1}x)<n+1}\exp(-2\alpha l(y^{-1}x)^{\beta })\\
				%&\le \left \| \xi  \right \|^{2}\sum_{n=0}^{\infty } \left( \left | B(e,n+1) \right | -\left | B(e,n) \right |  \right)  \exp(-2\alpha l(y^{-1}x)^{\beta })\\
				&\le \left \| \xi  \right \|^{2}\sum _{n=0}^{\infty}\exp(-2\alpha n^{\beta })\left | B(e,n+1) \right |  
				\le C\left \| \xi  \right \|^{2} \sum _{n=0}^{\infty} \exp(-2\alpha n^{\beta })\exp(\alpha^{'} (n+1)^{\beta^{'} })\\
				&\le C\left \| \xi  \right \|^{2}[
				C+\sum _{n=1}^{\infty} \exp(-\alpha(2n^{\beta}-(n+1)^{\beta}))]\le C  \left \| \xi  \right \| ^{2},
			\end{split}
		\end{align*}
		where $2n^{\beta}-(n+1)^{\beta}>0$  for $n\ge 1$.\\
		Then, we derive 
		\begin{align*}
			\left \| \phi \xi  \right \|^{2}  	&\le C\left \| \xi  \right \|^{2}[\sup _{x} \sum_{y}\left |\phi (x,y) \right | ^{2}\exp(2\alpha l(y^{-1}x)^{\beta })]\\
			&\le C\left \| \xi  \right \|^{2} [\sum_{z\in G} \sup_{\left \{ x,y\in G,y^{-1}x=z \right \} }\left |\phi (x,y) \right | ^{2}\exp(2\alpha l(z)^{\beta })]\\
			&\le C \left \| \phi  \right \| _{\alpha ,\beta }^{2} \left \| \xi  \right \| ^{2}.
		\end{align*}
		Hence, we get $\left \| \phi   \right \| \le  C \left \| \phi  \right \| _{\alpha ,\beta }$. 
		By virtue of the boundedness of $\phi$
		, we set 
		\begin{equation*}
			\phi _{n}(x,y)=\left\{
			\begin{aligned}
				\phi(x,y) & , & \text {if}~ l(y^{-1}x)\le n;\\
				0 & , & otherwise.
			\end{aligned}
			\right.
		\end{equation*}
		Then , we get 
		\begin{align*}
			\left \| \phi -\phi_{n} \right \| _{\alpha,\beta }^{2}
			%&\le \sum_{k=n+1}^{\infty } \sum _{k\le l(y^{-1}x)<k+1}\Big | \sup_{\left \{ x,y\in G,y^{-1}x=z \right \} } \phi(x,y) \Big | ^{2}\exp(2(2\alpha+1) l(y^{-1}x)^{\beta })\exp(-2(\alpha+1) l(y^{-1}x)^{\beta })\\
			&\le\left \| \phi  \right \| _{2\alpha+1,\beta}^{2}\sum_{k=n+1}^{\infty } \sum _{k\le l(z)<k+1}\exp(-2(\alpha+1) l(z)^{\beta })\\
			&\le\left \| \phi  \right \| _{2\alpha+1,\beta}^{2}
			\sum_{k=n+1}^{\infty }  \left | B(e,k+1) \right |  \exp(-2(\alpha+1) k^{\beta })\\
			%&\le C\left \| \phi  \right \| _{2\alpha+1,\beta}^{2}\sum_{k=n+1}^{\infty } \exp(\alpha^{'} (k+1)^{\beta' }) \exp(-2(\alpha+1) k^{\beta })\\
			&\le C\left \| \phi  \right \| _{2\alpha+1,\beta}^{2}\sum_{k=n+1}^{\infty } \exp(-\alpha(2k^{\beta}-(k+1)^{\beta})-2k^{\beta})\\
			&\le C\left \| \phi  \right \| _{2\alpha+1,\beta}^{2}\sum_{k=n+1}^{\infty } \exp(-2k^{\beta })<\varepsilon.
		\end{align*}
		Therefore, we conclude
		$\left \| \phi -\phi_{n} \right \|<\varepsilon,$ and $R^{\infty }(G)\subseteq C_{u}^{*}(G)$.
	\end{proof}
	We provide the following lemma which is important to prove $R^{\infty }(G)$ is a spectral invariant subalgebra in Theorem \ref{spectral subgrow}.
	\begin{lemma}
		Taking $A=(a(x,y))_{x,y\in G}\in R_{\alpha,\beta}(G) $, let $ v(z)=v(z,I_{0} )$, $ w_{\left \{ x,y\in G,y^{-1}x=z \right \} }\\(x,y)=w(z,I_{0} )$. Then under the assumptions of the weights $w$ and $v$, we have
		\begin{align}\label{lemmamax}
			\left \| A \right \| _{\mathcal{B}^{2} }&\le \max (\sup_{x\in G}\sum _{y\in G}\left | a(x,y) \right |, \sup_{y\in G}\sum _{x\in G}\left | a(x,y) \right |)\notag\\
			&\leq  \left \| A \right \|_{\alpha,\beta}\left \| w^{-1}  \right \|_{2}\le C\left \| A \right \|_{\alpha,\beta}\left \| vw^{-1}  \right \|_{2},~ C>0.
		\end{align}
	\end{lemma}
	\begin{proof}
		By the definition of the operator norm, we have
		\begin{align*}
			\left \| A \right \| _{\mathcal{B}^{2} }&
			\le \sup_{\left \| \xi  \right \|_{2} =1}[ \sum _{x\in G} ( \sum _{y\in G} \left | a(x,y) \right |^{2}  ) (  \sum _{y\in G} \left | \xi (y) \right | ^{2}) ] ^\frac{1}{2}\\
			&\le \sup_{y} (\sum _{x\in G}\left | a(x,y) \right |) ^\frac{1}{2}\cdot \sup_{x}( \sum _{y\in G}\left | a(x,y) \right |) ^\frac{1}{2}\\
			&\le \max (\sup_{x\in G}\sum _{y\in X}\left | a(x,y) \right |, \sup_{y\in G}\sum _{x\in X}\left | a(x,y) \right |).
		\end{align*}
		Using the Cauchy-Schwarz Inequality we get 
		\begin{align*}
			\left \| A \right \| _{\mathcal{B}^{2} }
			&\le (\sum _{z\in G}  (\sup_{\left \{ x,y\in G,y^{-1}x=z \right \} } \left | a(z) \right | w(z))^2 )^\frac{1}{2} \cdot(\sum_{z\in G}  w^{-2}(z))^\frac{1}{2}\notag\\
			&%=(\sum _{z\in G}  ( \tilde{a}(z)w(z))^2 )^\frac{1}{2} (\sum _{z\in G}  w^{-2}(z))^\frac{1}{2}
			\le \left \| A \right \|_{\alpha,\beta}\left \| w^{-1} \right \| _2\le C\left \| A \right \|_{\alpha,\beta}\left \| vw^{-1} \right \|_2,
		\end{align*}
		which implies  (\ref{lemmamax}).
	\end{proof}
	
	\begin{theorem}\label{spectral subgrow}
		Let $G$ be a countable discrete group with a proper length function $l$, satisfying  property $P$. Then, the  algebra $ R^{\infty }(G)$ is a spectral invariant dense subalgebra of the uniform Roe algebra $C_{u}^{*}(G)$.
	\end{theorem}
	\begin{proof}
		%$R^{\infty }(G)$ is  spectral invariant means, for any $A\in  R^{\infty }(G)$,  $A^{-1}\in  R^{\infty }(G)$.
		To establish the spectral invariance property, the proof proceeds in two key steps: (i) deriving an estimate for the 
		$n$-th power of $A$, and (ii) verifying the inverse-closed property.
		%Let $z = y^{-1}x$, for any $ x, y \in  G$. 	 Consider the weight  $w$ and $v$ of following form
		%\[ w(z)=\exp(\alpha \rho_{l }(z)^{\beta})~   \alpha>0 ~,~ 0<\beta<1 0~\text{and}~v(z)=1.\]
		
		\textbf{Step 1.} For any $A\in R^{\infty}(G)$ and $n\ge 1$, the following  inequality holds
		\begin{align}\label{estim}
			\left \| A^{n} \right \|_{\alpha,\beta } \le C\left ( C\left \| A \right \|_{\alpha,\beta }\left \| A \right \|_{\mathcal{B} ^{2}}^{-1}   \right )^{\frac{1+\theta }{\theta }n^{\log_{2}{(1+\theta )} } } (\left \| A \right \|_{\mathcal{B} ^{2}})^{n}.
		\end{align}
		%	It follows from (\ref{step1}) that
		%{ \bf Step 1.} The following  inequality holds
		
		%\begin{align}\label{step1}
		%\left \| A ^2 \right \| _s\le C \left \| A \right \|_^{1-\theta }  \left \| A \right \|_{\mathcal{B}^2 }^{1-\theta } \quad\text {for~all}~ A\in H_{l}^{\infty}\left ( G \right )~\text{and}~C>0.
		%\end{align}
		\noindent 
		Let $A=(a(x,y))_{x,y\in G}\in R_{\alpha,\beta}(G)$, and $A^2=(c(x,y))_{x,y\in G}$. For simplify, we denote
		$ w(x)=w(x,I )$, $v(x)=v(x,I ),\tilde{a}(x)=\sup_{y\in G } 
		\left | a(yx,y) \right |$ and $\tilde b(x)=\sup_{y\in G}\left | b(yx,y) \right |$, where $I$ is the identity element of the discrete group $G$.
		
		\noindent Since $w(z)=\exp(\alpha \rho_{l }(z)^{\beta})$ is an admissible weight by Theorem \ref{thm4.2}, it follows that for any $x\in G$,
		\begin{align}
			\label{6.10}
			\tilde{c}w(x) &\le C\sup_{y\in G}\sum _{z\in G}(\left | (av)(yx,z) \right | \left | (aw) (z,y)\right |+\left | (aw)(yx,z) \right | \left | (av) (z,y)\right |).
		\end{align}
		Moreover, we get
		\begin{align}\label{norm}
			&\sup_{y\in G}\sum _{z\in G}\left | a(yx,z) \right | v\left ( yx,z \right )\left | a(z,y) w(z,y)\right |\notag\\
			&\leq\sup_{y\in G}(\sum _{z\in G}(v(yx,z)\chi _{\rho (yx,z)<\tau }\left | a(z,y) \right | w(z,y))^2 )^\frac{1}{2}  ( \sum _{z\in G} \left | a(yx,z) \right |^2 )^\frac{1}{2} \notag\\
			&\quad+\sup_{y\in G}\sum _{z\in G}\left | a(yx,z) \right | v\left ( yx,z \right  )\chi _{\rho (yx,z)\ge \tau}\left | a(z,y) w(z,y)\right |\notag\\
			&\le (\sum_{z\in G}( v(z^{-1}x)\chi_{\rho (x,z)<\tau}\left | \tilde a(z) \right |  w(z))^{2})^\frac{1}{2} \left \| A \right \|_{\mathcal{B}^{2} }\notag\\
			&\quad +\sum_{z\in G} (\tilde {a}(z^{-1}x) v(z^{-1}x)\chi_{\rho (x,z)\ge 
				\tau}\left | \tilde {a}(z) \right |  w(z)),
		\end{align}
		%where we consider $a(yx,\cdot )$ as a sequence $\left [ a(yx,z) \right ]_{z\in G}$, then it is easy to get$ ( \sum _{z\in G} \left | a(yx,z) \right |^{2}  )^\frac{1}{2} \le \left \| A \right \|_{\mathcal{B}^{2} }$ by the operator norm.
		
		\noindent For any $\tau \ge 1$, based on (\ref{w2}) and (\ref{norm}) we obtain
		\begin{align}\label{6.8}
			&  \|( \sup_{y\in G}\sum _{z\in G}\left | (av)(yx,z) \right | \left | (aw) (z,y)\right | )_{x\in G} \|_{2}\notag \\
			&\leq \inf_{\tau \ge 1}(\left \| \tilde {a} w \right \|_{2}\left \|  v\chi _{B(I_{0},\tau )} \right \|_{2}\left \| A \right \|_{\mathcal{B} ^{2}} +\left \| \tilde {a} w \right \|_{2}\left \| (\tilde {a} v)\chi _{X\setminus B(I_{0},\tau )} \right \|_{2}) \notag\\
			&\leq C\left \| A \right \|_{\alpha,\beta }\left \| A \right \| _{\mathcal{B} ^{2}} \inf_{\tau \ge 1}(\left \|  v\chi _{B(I_{0},\tau )} \right \|_{2} +\frac{\left \| A \right \|_{\alpha,\beta } }{\left \| A \right \|_{\mathcal{B}^{2}} } \left \| ( v w^{-1})\chi _{X\setminus B(I_{0},\tau )} \right \|_{2})\notag \\
			&\leq C\left \| A \right \|_{\alpha,\beta }\left \| A \right \| _{\mathcal{B} ^{2}}D(\frac{\left \| A \right \|_{\alpha,\beta } }{\left \| A \right \|_{\mathcal{B} ^{2}} })^{\theta } =C\left \| A \right \|_{\alpha,\beta }^{1+\theta }\left \| A \right \| _{\mathcal{B} ^{2}}^{1-\theta }.
		\end{align}
		Similarly, we get
		\begin{equation}
			\label{6.9}
			\|( \sup_{y\in G}\sum _{z\in G}\left | (aw)(yx,z) \right | \left | (av) (z,y)\right | )_{x\in G} \|_{2}\le C\left \| A \right \|_{\alpha,\beta}^{1+\theta }\left \| A \right \| _{\mathcal{B} ^{2}}^{1-\theta }.
		\end{equation}
		Then, combining (\ref{6.10}), (\ref{6.8}) and (\ref{6.9}), we get 	\begin{align*}%\label{step1}
			\left \| A ^2 \right \| _{\alpha,\beta}\le C \left \| A \right \|_{\alpha,\beta}^{1-\theta }  \left \| A \right \|_{\mathcal{B}^2 }^{1-\theta } \quad\text {for~all}~ A\in R^{\infty}\left ( G \right )~\text{and}~C>0,
		\end{align*}
		which indicates
		\begin{equation*}
			\left \| A^{2n} \right \|_{\alpha,\beta } 
			\le  D\left \| A^{n} \right \|_{\alpha,\beta }^{1+\theta }\left \| A \right \| _{\mathcal{B} ^{2}}^{n( 1-\theta) }~\text{and}~	\left \| A^{2n+1} \right \|_{\alpha,\beta } \le D_{1}\left \| A \right \|_{\alpha,\beta } \left \| A^{n} \right \|_{\alpha,\beta }^{1+\theta }\left \| A \right \| _{\mathcal{B} ^{2}}^{n(1-\theta) }
		\end{equation*}
		for some positive constants $D$ and $D_1$. Without loss of generality, we assume $D_1\geq D$, and $D_1\|A\|_{\mathcal{B}^2}\geq 1$ and define the sequence $\left \{ b_{n} \right \} $ by
		\begin{equation}\label{Dbn}
			b_{n}=D_{1}^{\frac{1}{\theta } }\left \| A^{n} \right \| _{\alpha,\beta }\left \| A \right \|_{\mathcal{B} ^{2}}^{-n} \quad\text{for~all}~ n\ge1,
		\end{equation}
		satisfying
		\begin{equation*}
			b_{2n}\le b_{n}^{1+\theta } \quad \text {and} \quad b_{2n+1}\le b_{1}b_{n}^{1+\theta }    \quad\text{for~all}~ n\ge 1,
		\end{equation*}
		which implies that
		\begin{align}\label{bn}
			b_{n}\le b_{1}^{\sum_{i=0}^{k } \varepsilon _{i}(1+\theta) ^{i}}\quad \text{for}~n=\sum_{i=0}^{k}\varepsilon _{i}2^{i}, \varepsilon _{i}\in \left \{ 0,1 \right \}.
		\end{align}
		%	for $ n=\sum_{i=0}^{k}\varepsilon _{i}2^{i} $, $\varepsilon _{i}\in \left \{ 0,1 \right \} ,0\le i\le k$. 
		For the index 
		$\sum_{i=0}^{k } \varepsilon _{i}(1+\theta) ^{i}$ in (\ref{bn}), we have
		\begin{align}\label{eso}
			{\sum_{i=0}^{k } \varepsilon _{i}(1+\theta) ^{i}}\le {\sum_{i=0}^{k } (1+\theta) ^{i}}%=\frac{(1+\theta)(1+\theta)^{k}-1}{\theta}
			\le\frac{(1+\theta)}{\theta}(1+\theta)^{k}.
		\end{align}
		Since $n=\sum_{i=0}^{k}\varepsilon _{i}2^{i}\le\sum_{i=0}^{k}2^{i}\le 2^{k+1}-1$, we have $k= \lfloor \log_2(1+n)\rfloor-1$, where $\lfloor x\rfloor$ denotes the greatest integer less than or equal to $x$. 
		
		\noindent It follows from (\ref{Dbn})- (\ref{eso}) that
		\begin{align*}
			D_{1}^{\frac{1}{\theta } }\left \| A^{n} \right \| _{\alpha,\beta }\left \| A \right \|_{\mathcal{B} ^{2}}^{-n}&\le \Big ( D_{1}^{\frac{1}{\theta } }\left \| A \right \| _{\alpha,\beta }\left \| A \right \|_{\mathcal{B} ^{2}}^{-1} \Big ) ^{{\sum_{i=0}^{k } \varepsilon _{i}(1+\theta) ^{i}} }\notag\\
			&\le \Big ( D_{1}^{\frac{1}{\theta } }\left \| A \right \|_{\alpha,\beta } \left \| A \right \|^{-1} _{\mathcal{B} ^{2}} \Big ) ^{\frac{(1+\theta)}{\theta}n^{\log_2(1+\theta)}}.
		\end{align*}
		Therefore, we get the desired estimate (\ref{estim}).
		
		{\bf Step 2.} For any $A\in  R^{\infty }(G)$, we have $A^{-1}\in  R^{\infty }(G)$.
		
		%Finally, we prove that $H_{l}^{\infty } \left ( G \right )$ is spectral invariant in $C_{u}^{*}  \left ( G \right )$.\\
		For any $A=(a(x,y))_{x,y\in G}\in  R^{\infty }(G)$, we define its transpose $A^{*}=(\overline{a(y,x)} )_{x,y\in G}$. Then, we have $A^{*}A\in  R^{\infty }(G)$ and  $\left \| A \right \|_{\alpha,\beta }= \left \| A^{*} \right \|_{\alpha,\beta }$. 
		\noindent Moreover, we define the matrix $B\in \mathcal{B} (\ell^{2}(G))$ by 
		$$	B=I-\frac{2 A^{*}A}{C_2+C_1 }.$$
		Since $A^{*}A$ is a positive operator, there exist constants $C_1>0$, $C_2>0$ such that $C_{1}I\le A^{*}A\le C_{2}I$ and
		\begin{equation}\label{2}
			\left \| B \right \|_{\mathcal{B} ^{2}}\le\frac{C_2-C_1}{ C_2+C_1} <1  \quad \text{and}\quad  \left \| B \right \|_{\alpha,\beta }<\infty.
		\end{equation}
		It follows from (\ref{estim}) and (\ref{2}) that  
		\begin{align*}
			\left \|( I-B)^{-1} \right \|_{\alpha,\beta } %&= \Big \| \sum_{n=0}^{\infty }B^{n}  \Big \|_{\alpha,\beta }\leq \sum_{n=0}^{\infty}\|B^n\|_s \\
			\le \sum_{n=0}^{\infty} C\Big ( C\left \| B \right \|_{\alpha,\beta } \left \| B \right \|_{\mathcal{B}^{2}}^{-1} \Big  )^{\frac{1+\theta }{\theta }n^{\log_{2}{(1+\theta )} } } (\left \| B \right \|_{\mathcal{B}^{2}}) ^{n}<\infty ,
		\end{align*}
		which implies that $(A^{*}A)^{-1}\in R^{\infty}(G)$.
		Consequently, we deduce that $A^{-1}\in  R^{\infty }(G)$, as $A^{-1}=(A^{*}A)^{-1}A^{*}$.\par
		Thus, $ R^{\infty }(G)$ is spectral invariant in $C_{u}^{*}  \left ( G \right )$.
	\end{proof}
	Theorem \ref{spectral subgrow} still holds when the group $
	G$ has polynomial growth.
	\begin{corollary}
		If $G$ has polynomial growth, the algebra $ R^{\infty }(G)$ is a spectral invariant  subalgebra of the uniform Roe algebra $C_{u}^{*}(G)$.
	\end{corollary}
	\begin{proof}
		If $G$ has polynomial growth, we know that $G$ also satisfies property $P$. The weight $w(x,y)=\exp(\alpha \rho_{l }(x,y)^{\beta})$ is a $(2,r)$-admissible weight.
		Therefore, by applying the similar argument used in the proof of Theorem \ref{spectral subgrow}, we conclude that $ R^{\infty }(G)$ is spectrally invariant.
	\end{proof}
	
	\section*{Acknowledgements}
	This research of Xian-Jin Wang was partially supported by NSFC General  Projects No.12071183 and No.11771061. 
	The authors declare that they have no conflict of interest.
	The authors declare that they have no conflict of interest. Data sharing is not applicable as no datasets were generated or analyzed in this study.
	
	\bibliographystyle{abbrv}
	\bibliography{ref1}	
\end{document}